\documentclass[a4paper,12pt]{amsart}
\usepackage{amsmath, latexsym, amsfonts, amssymb, amsthm, amscd}

\newtheorem{theorem}{Theorem}
\newtheorem{corollary}{Corollary}

\newtheorem{lemma}{Lemma}

\newcommand{\p}{\mathbb{P}}

\newcommand{\R}{\mathbb{R}}

\newtheorem*{rem*}{Remark}

\def\({\left(}
\def\){\right)}
\def\[{\left[}
\def\]{\right]}
\def\<{\langle}
\def\>{\rangle}

\title{Supremum distribution of Bessel process of drifting Brownian motion}
\author{A.~Py\'c, G.~Serafin and T. \.Zak}
\address{Institute of Mathematics and Computer Science, Wroc{\l}aw University of
Technology, Wybrze\.ze Wyspia\'nskiego 27, 50-370 Wroc{\l}aw, Poland.}
\email{andrzej.pyc@pwr.edu.pl, grzegorz.serafin@pwr.edu.pl, tomasz.zak@pwr.edu.pl}


\begin{document}

\begin{abstract}
Let $(B_t^{(1)},B_t^{(2)},B_t^{(3)}+\mu t)$ be a three-dimensional Brownian motion with drift $\mu$, starting at the origin.
Then $X_t=\|(B_t^{(1)},B_t^{(2)},B_t^{(3)}+\mu t)\|$, its distance from the starting point, is a diffusion with many applications. 
We investigate the distribution of the supremum of $(X_t)$, give an infinite-series formula for its density and an exact estimate by elementary functions.
\end{abstract}

\maketitle

\section{introduction}
In his famous paper (\cite{w}, see also \cite{rw}, page 436) David Williams showed how one can decompose the paths of a transient one-dimensional diffusion at its maximum (or minimum).
One of the best-known examples of such decomposition is that of $B(t)+\mu t$, Brownian motion with a positive drift $\mu$, as a Brownian motion with a negative drift $B(t)-\mu t$ and a diffusion with a generator  
\begin{equation}\label{generator}
\Delta_{\mu} = \frac12 \frac{d^2}{dx^2} +  \mu \coth(\mu x) \frac{d}{d x}.
\end{equation}
Namely, for $\mu >0$ let $B(t)-\mu t$ be a Brownian motion with constant drift $-\mu$, $Z_t$ a diffusion with generator (\ref{generator}), started at zero, and let $\gamma$ be a random variable with exponential distribution with parameter $\frac{1}{2\mu}$. Assume that $B_t$, $Z_t$ and $\gamma$, defined on the same probability space, are independent. Define $\tau= \inf\{t>0:\ B_t-\mu t=-\gamma\}$.
Then the following process
$$X_t=
\left\{
\begin{array}{ll}
B_t-\mu t, &\ \ 0\le t\le \tau,\\
Z_{t-\tau}-\gamma, & \ \  t>\tau,
\end{array}
\right.
$$
is $B_t+\mu t$, a Brownian motion with positive drift $\mu$. Williams also showed that $Z_t$ can be viewed as a Brownian motion with drift conditioned to hit infinity before zero.

\medskip

Later on a process  with generator (\ref{generator}), known as a hyperbolic Bessel proces (\cite{ry}, p.\, 357) 
or as a Bessel process of drifting Brownian motion and denoted BES$(3,\mu)$ (\cite{pr}, \cite{g}),  
appeared also in many papers as a drifting Brownian motion conditioned to not hit zero. Recently, 
it was proved in \cite{agz} that it can be obtained as a deterministic involution of Brownian motion with drift $\mu$.
If $\mu=1$ then $(Z_t)$ is a radial part of a hyperbolic Brownian motion in three-dimensional hyperbolic space.

\medskip

The transition density function of $(Z_t)$ is well-known, compare \cite{pr} or \cite{g} but to our best knowledge, 
distribiution of different functionals of this process have not been investigated yet. In this paper
we investigate process $(Z_t)$ killed on exiting interval $(0,\, r_0)$ and give a formula describing distribution 
of $M_t=\sup_{s\le t} Z_s$, the supremum of the process $(Z_t)$.
Because the formula is given as an infinite series, we give its exact estimate using elementary functions.
Moreover, our method of estimation applied to a function $ss_y(v,t)$ used in a handbook by Borodin and Salminen \cite{bs}
give very precise estimate of this function (cf. Remark after Theorem 9).

\section{Transition density of a Bessel process of drifting Brownian motion, killed on exiting interval $(0,\ r_0)$}
Let $Z(t)$ be a diffusion on $(0,\infty)$ generated by operator (1.1) with $\mu>0$.
The speed measure of this diffusion is equal to $m(dy)=\sinh^2 (\mu y)\, dy$, hence with respect to $m(dy)$, the transition density of $Z(t)$ has the following form (cf. \cite{pr}):
\begin{equation}\label{densfree}
p(t;x,y)=\frac{e^{-\mu^2t/2}\left(e^{-(y-x)^2/(2t)}-e^{-(y+x)^2/(2t)}\right)}{\sqrt{2\pi t}\sinh(\mu x) \sinh(\mu y)}.
\end{equation}
By $(Z^{r_0}_t)$ we will denote the process killed on exiting $(0,\, r_0)$. For positive $\mu$, the process starting from $x>0$ cannot reach zero and drifts to infinity so that almost all trajectories will be killed at $r_0$. Even if $(Z_t)$ starts from zero, with probability one it will never visit zero again.

Transition density $p^{r_0}(t;x,y)$ of $(Z^{r_0}_t)$, with respect to  $m(dy)$, is a solution of the following Dirichlet problem:
\begin{equation}\label{system}
\begin{cases} \frac{\partial}{\partial t} p^{r_0}(t;x,y) = \Delta_{\mu} p^{r_0}(t;x,y) & t>0,\ x \in (0,r_0),\ y \in (0,r_0), \\ p^{r_0}(t;x,r_0) = 0 & t>0,\ x \in (0,r_0), \\ \lim_{t \to 0} p^{r_0}(t;x,y) = \delta_x(y) &x \in (0,r_0),\ y \in (0,r_0). \end{cases}
\end{equation}
Because of killing, $p^{r_0}(t;x,r_0)=0$ for all $t>0$. Moreover, if $\mu>0$ then by (\ref{densfree}) we have $\limsup_{y\to 0}p(t;x,y)<\infty$ for all $t,\, x>0$ and, because $p^{r_0}(t;x,y) \le p(t;x,y)$, hence $\limsup_{y\to 0}p^{r_0}(t;x,y)<\infty$.
 
We will use the separation variable technique used in mathematical physics. Suppose that 
$$
p^{r_0}(t;x,y) = Y(y) T(t).
$$ 
Then, by (\ref{generator}), the first equation of the system (\ref{system}) takes on the form
\begin{equation}\label{Dirichlet}
Y(y) T'(t)  = \left(\frac{1}{2} Y''(y) + \mu \coth(\mu y) Y'(y)\right) T(t).
\end{equation}
We add $\frac{1}{2}(\lambda^2+\mu^2)T(t)Y(y)$ to both sides of (\ref{Dirichlet}) and get
$$
Y(y) [T'(t) + \frac{1}{2}(\lambda^2+\mu^2)T(t)] = T(t) [\frac{1}{2}Y''(y) + \mu \coth(\mu y) Y'(y) + \frac{1}{2}(\lambda^2+\mu^2)Y(y)].
$$
The solution of the differential equation
$
T'(t) + \frac{1}{2}(\lambda^2+\mu^2)T(t) = 0
$
is $T(t)=c_1e^{\frac{-(\lambda^2+\mu^2)t}{2}}$.
If in the second equation
$$
Y''(y) + 2 \mu \coth(\mu  y) Y'(y) + (\lambda^2+\mu^2)Y(y) = 0
$$
we substitute  $Y(y)=u(y)/\sinh(\mu y)$, we get
$$
\frac{1}{\sinh \mu y} (u''(y)+\lambda^2 u(y))=0.
$$
By the above discussion  we know that $\limsup_{y\to 0} Y(y)<\infty$ and $Y(y)=u(y)/\sinh(\mu y)$, hence $\lim_{y\to 0} u(y)=0$.
This implies that with the boundary conditions $u(0)=u(r_0)=0$ equation  $u''(y)+\lambda^2 u(y)=0$
is a special case of the general Sturm -- Liouville problem \cite{H}.
It is well known (cf. \cite{H}, Theorem 4.1 and Exercise 4.2, page 337) that this Sturm -- Liouville problem for $u(y)$ has a solution if and only if $\lambda=\frac{n\pi}{r_0}$, $n=1,2,...$
and this solution (up to a multiplicative constant) is given by $u_n(y)=\sin(n\pi y/r_0)$. 
Thus we may expand $p^{r_0}(t;x,y)$ as
\begin{equation}\label{series1}
p^{r_0}(t;x,y) = \sum_{n=1}^\infty \left[ \left(\frac{a_n(x)}{\sinh (\mu y)} \sin\left(\frac{n \pi y}{r_0}\right)\right) \exp\left(\frac{-(n^2 \pi^2 / r_0^2 +\mu^2)t}{2} \right)  \right].
\end{equation}

To determine coefficient $a_n(x)$  let us multiply the equation (\ref{series1}) by $a_k(x)\frac{\sin(\frac{k \pi y}{r_0})}{\sinh (\mu y)}$ and integrate the product over $(0,r_0)$ with respect to the measure $\sinh^2 (\mu y) dy$:
\begin{eqnarray}\label{series2}
&\int_0^{r_0} p^{r_0}(t;x,y) a_k(x)\frac{\sin(\frac{k \pi y}{r_0})}{\sinh (\mu y)}\sinh^2 (\mu y) dy = \\ \nonumber 
= &\int_0^{r_0}  \sum_{n=1}^\infty \left[ \left(\frac{a_n(x) a_k(x)}{\sinh^2 (\mu y)} \sin\left(\frac{n \pi y}{r_0}\right) \sin\left(\frac{k \pi y}{r_0}\right)\right)\exp\left(\frac{-(n^2 \pi^2 / r_0^2 +\mu^2)t}{2} \right)  \right] \sinh^2 (\mu y) dy.
\end{eqnarray}
Observe that
$$
\int_0^{r_0} a_n(x) a_k(x) \sinh^2 (\mu y) \frac{\sin(\frac{n \pi y}{r_0})}{\sinh(\mu y)} \frac{\sin(\frac{k \pi y}{r_0})}{\sinh (\mu y)} dr = a_n^2(x) \frac{r_0}{2} \delta_k^n,
$$
where $\delta_k^n$ is the Kronecker delta. Hence the right-hand side of (\ref{series2}) is equal to
$$
\frac{r_0}{2} a_n^2(x) \exp \left(\frac{-\mu^2 t}{2} - \frac{n^2 \pi^2 t}{2 r_0^2} \right).
$$
Now, let $t \to 0$. We have $\exp\left(-\frac{(n^2 \pi^2/r_0^2 + )t}{2} \right) \to 1$ and, using the third condition from (\ref{system}) for the left-hand side of (\ref{series2}), we get
$$
\frac{a_k(x) \sin\left(\frac{k \pi x}{r_0}\right)}{\sinh (\mu x)} = \frac{r_0}{2} a_k^2(x),
$$
hence
$$
a_k(x) = \frac{2 \sin\left(\frac{k \pi x}{r_0}\right) }{r_0 \sinh (\mu x)}. 
$$
To sum up, we have just proved the following theorem.

\medskip

\begin{theorem}
Transition density (with respect to the measure $\sinh^2 (\mu y)\, dy$) of the Bessel process of drifting Brownian motion, starting from the point $x \in (0, r_0)$ and killed at $r_0$, is given by the following formula

\begin{equation}\label{series}
p^{r_0}(t;x,y) = \sum_{n=1}^\infty \left[ \left(\frac{2 \sin\left(\frac{n \pi x}{r_0}\right) \sin\left(\frac{n \pi y}{r_0}\right) }{ r_0 \sinh (\mu x) \sinh (\mu y)}\right)  \exp\left(-\frac{(n^2 \pi^2  / r_0^2 +\mu^2)t}{2} \right)  \right].
\end{equation}

\end{theorem}

\medskip

We will give another representation of the transition density, using the Poisson summation formula (cf. \textbf{13.4} in \cite{Z}): for any function $g$ absolutely integrable on $(-\infty,\ \infty)$
\begin{equation} \label{poisson}
\sum_{n=- \infty}^\infty g(n) = \sum_{k=- \infty}^\infty \int_{- \infty}^\infty g(x) e^{-2 \pi i k x} dx.
\end{equation}
First, observe that series (\ref{series}) can be written as 
$$
p^{r_0}(t;x,y) = \frac{1}{2} \sum_{n = -\infty}^\infty \left[ \left(\frac{2 \sin\left(\frac{n \pi x}{r_0}\right) \sin\left(\frac{n \pi y}{r_0}\right) }{r_0 \sinh (\mu x) \sinh (\mu y)}\right)  \exp\left(-\frac{(n^2 \pi^2  / r_0^2 +\mu^2)t}{2} \right)  \right].
$$ 
For the Poisson summation formula we need to compute the Fourier transform of any term of the above series, taking $n$ as a variable.
Put $z$ instead of $n$, then we have to evaluate the following integral:
$$
\int_{-\infty}^{\infty} \sin\left(\frac{z \pi x}{r_0}\right) \sin\left(\frac{z \pi y}{r_0}\right) 
e^{\frac{-z^2 \pi^2 t}{2 r_0^2}} e^{-2 \pi i  k z}  dz.
$$ 
In order to compute it, we use trigonometric identity  $e^{-2 \pi i  k z} = \cos(2\pi k z)-i \sin(2\pi k z)$ and observe that the integral with $\sin(2\pi k z)$ is zero, hence we have to compute only
$$
\int_{-\infty}^{\infty}  \sin\left(\frac{z \pi x}{r_0}\right) \sin\left(\frac{z \pi y}{r_0}\right) \cos(2\pi k z)
e^{\frac{-z^2 \pi^2 t}{2 r_0^2}} dz.
$$ 
Now we use identity $2\sin ax \sin bx = \cos((a-b)x) -\cos((a+b)x)$ and a formula (3.898.2) from \cite{gr} to get:
$$
\int_{-\infty}^{\infty} \sin\left(\frac{z \pi x}{r_0}\right) \sin\left(\frac{z \pi y}{r_0}\right) 
e^{\frac{-z^2 \pi^2 t}{2 r_0^2}} e^{-2 \pi i  k z}  dz =  
$$
$$
= \frac{r_0}{2 \sqrt{2 \pi t}} \left(e^{\frac{-(y-x+2k r_0)^2}{2t}} - e^{\frac{-(y+x-2k r_0)^2}{2t}} +
e^{\frac{-(y-x-2k r_0)^2}{2t}} - e^{\frac{-(y+x+2k r_0)^2}{2t}} \right).
$$
The Poisson summation formula (\ref{poisson}) gives us the following:

\begin{theorem}
Transition density (with respect to the measure $\sinh^2 (\mu y)\, dy$) of the Bessel process of drifting Brownian motion,
starting from the point $x \in (0, r_0)$ and killed at $r_0$, is given by the following formula
$$
p^{r_0}(t;x,y) = \frac{e^{\frac{-\mu^2 t}{2}}}{\sqrt{2 \pi t} \sinh(\mu x) \sinh(\mu y)} \sum_{k=- \infty}^\infty 
\left[e^{\frac{-(y-x+2k r_0)^2}{2t}} - e^{\frac{-(y+x+2k r_0)^2}{2t}} \right].
$$
\end{theorem}

\medskip

\section{Exit time}

Let us consider $M_t = \sup_{s\le t} Z_s$, the supremum of the Bessel process of drifting Brownian motion.
The distribution of $(M_t)$ is closely related to the distribution of the time when the process $(Z_t)$ exits the interval $(0,\, r_0)$.
Recall that for $\mu>0$ process $(Z_t)$ exits $(0,\, r_0)$ at the point $r_0$.
For $r_0>0$ let us define $ \tau_{r_0} = \inf\{s: Z_s>r_0\}$.
Distribution of $(M_t)$ and the survival probability of the killed process $(Z^{r_0}_t)$ are related by the following formula:
$$
\mathbb{P}^x(M_t < r_0) = \mathbb{P}^x(\tau_{r_0}> t) = \int_0^{r_0} p^{r_0}(t;x,y) \sinh^2 (\mu y) dy.
$$

For fixed $\mu$ and $r_0$ series (\ref{series}) is uniformly convergent for $x,\, y\in [0,\, r_0]$, 
because there holds inequality $|\frac{\sin(n\pi x/r_0)\sin(n\pi y/r_0)}{\sinh(\mu x)\sinh(\mu y)}| \le c_{\mu, \, r_0} n^2$
and $\sum n^2\exp(-n^2\pi^2/(2r_0^2))<\infty$.
Thus we may integrate the series term by term. Integrating $n$-th term, we get (compare \cite{gr}, formula 2.671)
$$
\int_0^{r_0} \sin\left(\frac{n \pi y}{r_0}\right) \sinh(\mu y) dy = \frac{(-1)^{n+1} \pi r_0 n \sinh(\mu r_0) }{(n^2 \pi^2 + \mu^2 r_0^2)},
$$
so that we may write the following:

\begin{theorem}
For $t,\, r_0>0$ and $x\in (0,\, r_0)$ the following formula holds
$$
\mathbb{P}^x(M_t < r_0) = \mathbb{P}^x(\tau_{r_0}> t) =
$$
\begin{equation}\label{supremum}
\sum_{n=1}^\infty \left[ \left(\frac{(-1)^{n+1}2 \pi n \sinh(\mu r_0)\sin(n \pi x /r_0)}{\sinh(\mu x) (n^2 \pi^2 + \mu^2 r_0^2)}\right)  \exp\left(-\frac{(n^2 \pi^2/r_0^2 + \mu^2 )t}{2} \right)  \right].
\end{equation}
\end{theorem}
 
\bigskip

If we differentiate the above series term by term with respect to $t$ we get
$$
\frac{\pi \sinh(\mu r_0)}{\sinh(\mu x) r_0^2} \exp\left(-\frac{\mu^2 t}{2} \right) \sum_{n=1}^\infty (-1)^{n}  n \sin(n \pi x /r_0) 
\exp\left(-\frac{n^2 \pi^2 }{2r_0^2} t \right).$$ 
Fixing $\varepsilon>0$ we can  observe that the series of the derivatives is convergent uniformly  for $t\in [\varepsilon,\ \infty)$
so that differentiation term by term is justified. Since $\mathbb{P}^x(\tau_{r_0} \in dt) = - \frac{\partial}{\partial t} \mathbb{P}^x(\tau_{r_0} > t)$, the exit time density is given by the following formula.

\begin{theorem}
For fixed $r_0>0$, $0<x<r_0$ and any  $t>0$
\begin{equation}\label{time_density_1}
\mathbb{P}^x(\tau_{r_0} \in dt)  =
\frac{\pi \sinh(\mu r_0)}{\sinh(\mu x) r_0^2} \exp\left(-\frac{\mu^2 t}{2} \right) \sum_{n=1}^\infty (-1)^{n+1}  n \sin(n \pi x /r_0) 
\exp\left(-\frac{n^2 \pi^2 }{2r_0^2} t \right).
\end{equation}
\end{theorem}

\medskip

In a similar way as we did it for the density of the killed process, using the Poisson summation formula
we can obtain another representation of the exit time density.
Note that the series in formula (\ref{time_density_1}) can be written down in the following form:

$$
\sum_{n=1}^\infty (-1)^{n+1}  n \sin(n \pi x /r_0) 
\exp\left(-\frac{n^2 \pi^2 }{2r_0^2} t \right)
= \frac{1}{2} \sum_{n= -\infty}^\infty (-1)^{n+1}  n \sin(n \pi x /r_0) 
\exp\left(-\frac{n^2 \pi^2 }{2r_0^2} t \right).
$$
In order to use (\ref{poisson}), we will compute the Fourier transform of its $n$-th term, 
taking $z$ in place of $n$. First we take only cosine of $\exp(-2 \pi i k z)$, next we integrate by parts and finally we use (\cite{gr}, formula 3.896.4) to get

\begin{equation}\label{fourier}
\int_{-\infty}^{\infty} z \sin\left(\frac{z \pi y}{r_0}\right) 
\exp\left(\frac{-z^2 \pi^2 t}{2 r_0^2}\right) \exp(-2 \pi i k z)  dz = 
\end{equation}
$$
\frac{r_0^2}{\sqrt{2}(\pi t)^{3/2}}
\left[(y+2kr_0) \exp\left( \frac{-(y+2kr_0)^2}{2t}\right) + (y-2kr_0) \exp\left(\frac{-(y-2kr_0)^2}{2t}\right)\right].
$$
Observe that
$$
\sin\left(\frac{n \pi (r_0-x)}{r_0}\right) = (-1)^{n+1} \sin(n \pi x /r_0),
$$
hence putting $y = r_0-x$ in (\ref{fourier}) and using Poisson formula, we get the second representation of the exit time density.

\medskip

\begin{theorem}
For fixed $r_0>0$, $0<x<r_0$ and any  $t>0$
$$
\mathbb{P}^x(\tau_{r_0} \in dt) 
= \frac{\sinh(\mu r_0) e^{-\mu^2 t/2}}{\sinh(\mu x)\sqrt{2 \pi} t^{3/2}} \sum_{k=-\infty}^{\infty}
(r_0 -x+ 2kr_0) \exp \left(\frac{-(r_0-x + 2kr_0 )^2}{2t}\right).
$$
\end{theorem}

\section{Mean exit time}

Now we want to compute $\mathbb{E}^x(\tau_{r_0})$ --- the mean exit time of $(Z_t)$ from $(0,\, r_0)$. We will use the formula
$$
\mathbb{E}^x(\tau_{r_0}) = \int_0^{\infty} t\,\mathbb{P}^x(\tau_{r_0} \in dt),
$$
so that we need to compute the integral
$$
\frac{\sinh(\mu r_0) }{\sqrt{2\pi} \sinh(\mu x)}
 \int_0^{\infty}  \frac{ e^{-\mu^2 t/2}}{\sqrt{ t}} \sum_{k=-\infty}^{\infty}
(r_0 -x + 2kr_0) \exp \left(\frac{-(r_0 -x + 2kr_0)^2}{2t}\right) dt.
$$ 
Integrating a single term we use formula (3.471(15)) from \cite{gr} and get
$$
\int_0^{\infty}\frac{1}{\sqrt{t}} \exp \left(\frac{-(r_0 + 2kr_0 - x)^2}{2t} - \frac{\mu^2 t}{2}\right)dt = \frac{\sqrt{2 \pi}}{\mu} \exp( - \mu|r_0+2kr_0-x|) dt,
$$
which gives
\begin{equation}\label{exit}
\mathbb{E}^x(\tau_{r_0}) = \frac{\sinh(\mu r_0)}{\mu \sinh(\mu x)}\sum_{k=-\infty}^{\infty} \left[(r_0-x + 2kr_0) \exp(-\mu |r_0-x+2kr_0|) \right]=\\
\end{equation}
$$
 \frac{\sinh(\mu r_0)}{\mu \sinh(\mu x)}\sum_{k=1}^{\infty}\left[ (r_0 -x + 2kr_0) \exp(\mu(-r_0+x-2kr_0))\right]+
 $$
 $$ \frac{\sinh(\mu r_0)}{\mu \sinh(\mu x)}\left[\sum_{k=1}^{\infty}\left[ (r_0 -x - 2kr_0) \exp(\mu(r_0-x -2kr_0))\right] + (r_0-x)\exp(-\mu(r_0-x)) \right].
$$
But
$$
(r_0-x)(e^{\mu(r_0-x)}+e^{\mu(x-r_0)})\sum_{k=1}^{\infty} \exp(-2\mu kr_0) = \frac{(r_0-x)}{e^{2\mu r_0}-1}\left(e^{\mu(r_0-x)}+e^{\mu(x-r_0)}\right)
$$
and
$$
2r_0(e^{\mu(x-r_0)}-e^{\mu(r_0-x)})\sum_{k=1}^{\infty} k \exp(-2 \mu kr_0) = \frac{2 r_0 e^{2 \mu r_0}}{(e^{2 \mu r_0}-1)^2}(e^{\mu(r_0-x)}-e^{\mu(x-r_0)}). 
$$
If we put them into (\ref{exit}), after some algebraic manipulation the formula for $\mathbb{E}^x(\tau_B)$ can be simplified a lot. Namely, we get
the following:

\medskip

\begin{theorem}
For any fixed $r_0>0$ and any starting point $x\in(0,\, r_0)$
$$
\mathbb{E}^x(\tau_{r_0}) = \frac{1}{\mu}\left(r_0\coth(\mu r_0)-x\coth(\mu x)\right).
$$
\end{theorem}

\bigskip

\section{Estimates}
Except for the last one, all the above formulas are given as series so that they are not convenient for computations or  applications. 
In this section we give exact approximations by elementary functions of the transition density of the killed process, of the killing time and of the density of the distribution of the supremum of the process.
Notation $ f\approx g$ means that there exist two absolute constants $c_1$ and $c_2$ such that for
all possible values of variables and parameters there holds $c_1 f < g <c_2 f$.

\medskip

For simplicity let us denote
$\gamma^{r_0}(t;x)=\p^x\(\tau_{r_0}\in dt\)$.
Recall that by Theorems 4 and 5 \ \ 
$$\gamma^{r_0}(t,x)=\rule{10cm}{0cm}$$
\begin{eqnarray}
&=\frac{\pi\sinh(\mu r_0)e^{-\mu^2t/2}}{\sinh(\mu x)r_0^2}\sum_{n=1}^{\infty}(-1)^{n+1} n\sin(n\pi x/r_0)\exp\(-\frac{n^2\pi^2}{2r_0^2}t\)\label{gamma2}\\
&=\frac{\sinh(\mu r_0)e^{-\mu^2t/2}}{\sinh(\mu x)\sqrt{2\pi}t^{3/2}}\sum_{n=-\infty}^{\infty}(r_0-x+2nr_0)\exp\(-\frac{(r_0-x+2nr_0)^2}{2t}\)\label{p}
\end{eqnarray}
and, by Theorem 2, 
$$
p^{r_0}(t;x,y)=\rule{10cm}{0cm}$$
\begin{eqnarray}
&=\frac{e^{-\mu^2t/2}}{\sinh(\mu x)\sinh(\mu y)\sqrt{2\pi t}}\sum_{n=-\infty}^{\infty}\[\exp\(-\frac{(y-x+2nr_0)^2}{2t}\)-\exp\(-\frac{(y+x+2nr_0)^2}{2t}\)\].\label{density}
\end{eqnarray}

\medskip

We will start with estimates in the particular case $r_0=1$. 
We also make separate calculations for $t\in(0,\frac14]$ and for $t\in[\frac14,\, \infty)$.

\bigskip

\begin{theorem}\label{smallt}
Let $r_0=1$. For $0<t\le \frac14$ and $0<x< 1$
$$
0,25 \le
\frac{ \gamma^1(t,x) }{ \frac{1}{\sqrt{2\pi} t^{3/2}}\frac{x(1-x)}{t+x}\frac{\sinh{\mu}}{\sinh (\mu x)}e^{-\mu^2 t/2-(1-x)^2/(2t)}}
\le 4,02.
$$
\end{theorem}

\begin{proof}  First we consider the case $0<x\le\frac12$ and use (\ref{p}).
It is enough to estimate only the part of $\gamma^1(t,x)$ consisting of a series, multiplied by $\frac{e^{(1-x)^2/(2t)}}{x}$, that is
the following quantity
$$
I = \frac{e^{(1-x)^2/(2t)}}{x}\sum_{k=-\infty}^{\infty}(2k+1-x)e^{-(2k+1-x)^2/(2t)}.
$$
Observe that we can group terms of the series: $k=0$ with $k=-1$,  $k=1$ with $k=-2$ and so on. In this way we get
\begin{equation}
I = \sum_{k=0}^{\infty} e^{-2k(k+1-x)/t}\frac{(2k+1-x)-(2k+1+x)e^{-2(2k+1)x/t}}{x}.\\\label{I}
\end{equation}
Now
$$
(t+x)I= (t+x)\sum_{k=0}^{\infty} e^{-2k(k+1-x)/t}\frac{(2k+1-x)-(2k+1+x)e^{-2(2k+1)x/t}}{x}=
 $$
  $$
 \sum_{k=0}^{\infty} e^{-2k(k+1-x)/t}\left(2k+1-x-t+\frac{t}{x}(2k+1)(1-e^{-2(2k+1)x/t})-(2k+1+x+t)e^{-2(2k+1)x/t}\right).
$$
Below we will estimate separately the term for $k=0$ and terms for $k\ge 1$.
The term for $k=0$ is equal to $g(x,t)=1-x-t+\frac{t}{x}(1-e^{-2x/t})-(1+x+t)e^{-2x/t}$. We will show that for $0<x\le 1/2$ and $0<t<1/4$
there holds $g(x,t)\le 2$.
Indeed, a derivative 
$$
\frac{\partial g(x,t)}{\partial x} = \frac{e^{-2x/t}}{tx^2}\left(2x^2(1+x)-tx(x(e^{2x/t}-1)-2)-t^2(e^{2x/t}-1)\right)
$$
is negative for $x>0$ and $t>0$, because $e^{2x/t}-1>\frac{2x}{t} +\frac{(2x)^2}{2t^2}$ and
this inequality implies that in the formula for the derivative, the quantity in the bracket is negative. 
Hence $g(x,t) \le \lim_{x\to 0} g(x,t)=2-2t$ and this is less than 2 for $t>0$.

In order to  estimate terms for $k\ge 1$, we use assumptions $0<t\le\frac14$, $0<x\le \frac12$ and inequality $1-e^{-x}<x$, valid for $x>0$ , and get
$$
(t+x)I\le 2+
\sum_{k=1}^{\infty} e^{-2k^2/t}\left(2k+1+\frac{t}{x}(2k+1)(1-e^{-2(2k+1)x/t})\right)\le
$$
$$
2+\sum_{k=1}^{\infty} e^{-8k^2}(2k+1)(4k+3) 
\le
2+21e^{-8}+\frac{14}{8}\sum_{k=2}^{\infty}8k^2e^{-8k^2}\le
$$
$$
2+21e^{-8}+\frac{14}{8}\sum_{n=32}^{\infty}ne^{-n} 
\le 2,01.
$$

To get an estimate from below,  we will use inequality $1-e^{-x}\ge \frac{x}{x+1}$, valid for $x>-1$. 
Using formula (\ref{I}), we get for $0<t\le \frac14$ and $0<x\le \frac12$: 
$$
I = \sum_{k=0}^{\infty} e^{-2k(k+1-x)/t}\frac{ -2x+(2k+1+x)(1-e^{-2(2k+1)x/t})}{x} \ge
$$
$$
\sum_{k=0}^{\infty} e^{-2k(k+1-x)/t}\frac{2(2k+1)(2k+1-x)-2t}{t+2(2k+1)x}\ge
\sum_{k=0}^{\infty} e^{-2k(k+1-x)/t}\frac{2k+1-x-t}{t+x}\ge
$$
$$
\frac{1}{t+x}\sum_{k=0}^{\infty} e^{-2k(k+1-x)/t} (2k+\frac14)\ge \frac1{4(t+x)},
$$
so that $(t+x) I \ge \frac14$. Both above estimates imply the following 
$$
 0,25
 \le  \frac{t+x}{x}e^{(1-x)^2/(2t)} \sum_{k=-\infty}^{\infty} (2k+1-x)e^{-(2k+1-x)^2/(2t)}\le 2,01.
$$
If, for $0<x\le \frac12$, we divide the middle term of the above inequality by $(1-x)\in[\frac12,\, 1)$, we must
multiply its left and right-hand sides by, respectively, 1 	and 2. This proves the theorem in this case. 
 
 \medskip

Now the proof for the case $\frac12\le x < 1$ and $0<t\le \frac14$.
Similarly as in the proof of the first case, we will examine the following quantity:
$$
\frac{e^{(1-x)^2/(2t)}}{1-x}\sum_{k=-\infty}^{\infty} (2k+1-x)e^{-(2k+1-x)^2/(2t)}=
\frac{1}{1-x}\sum_{k=-\infty}^{\infty} (2k+1-x)e^{-(2k^2+2k(1-x))/t}=
$$
$$
\frac{1}{1-x}\sum_{k=1}^{\infty} 2k\left( e^{-(2k^2+2k(1-x))/t} - e^{-(2k^2-2k(1-x))/t}\right)
+
\sum_{k=-\infty}^{\infty} e^{-(2k^2+2k(1-x))/t} =\frac{1}{1-x}A+B.
$$
First we estimate the series denoted by $B$: for $0< 1-x\le \frac12$ and $0<t\le \frac14$
$$
 B-1= \sum_{k=1}^{\infty} \left( e^{-(2k^2+2k(1-x))/t}+ e^{-(2k^2-2k(1-x))/t}\right) \le
\sum_{k=1}^{\infty} \left( e^{-2k^2/t}+ e^{-(2k^2-k)/t}\right) \le
$$
$$
\sum_{k=1}^{\infty} \left( e^{-8k^2}+ e^{-4k^2}\right) 
\le  \frac{1}{e^8}+\frac{1}{e^4}+\sum_{k=16}^{\infty} 2e^{-k}  <0,02,
$$
hence $1<B<1,02.$

Now we have to estimate $\frac{1}{1-x}A$. 
Observe that $A=\sum_{k=1}^{\infty} 2k\left( e^{-(2k^2+2k(1-x))/t} - e^{-(2k^2-2k(1-x))/t}\right)$ is negative  for $\frac12<x<1$ 
because all terms of the series are  negative. We will estimate the following positive quantity 
$$
\frac{-A}{1-x}=   \sum_{k=1}^{\infty} 2k\left( \frac{e^{-(2k^2-2k(1-x))/t} - e^{-(2k^2+2k(1-x))/t}}{1-x}\right)=
$$
$$
\sum_{k=1}^{\infty} 2k e^{-2k^2/t}\left( \frac{e^{2k(1-x)/t} - e^{-2k(1-x)/t}}{1-x}\right)=
4 \sum_{k=1}^{\infty} \frac{2k^2}{t} e^{-2k^2/t}\frac{\sinh(2k(1-x)/t)}{2k(1-x)/t}.
$$
For fixed $t>0$ and $k=1,2,3,...$ function $g(x)=\frac{\sinh(2k(1-x)/t)}{2k(1-x)/t}$ is decreasing for $\frac12 \le x < 1$ hence its maximal value is attained for $x=\frac12$ and is equal to $\frac{\sinh(k/t)}{k/t}=\frac{e^{k/t}-e^{-k/t}}{2k/t}$. Thus

$$
0<\frac{-A}{1-x} = 
4\sum_{k=1}^{\infty} \frac{2k^2}{t} e^{-2k^2/t}\frac{e^{k/t} - e^{-k/t}}{2k/t}\le
8 \sum_{k=1}^{\infty} k e^{-(2k^2-k)/t}\le 
8\sum_{k=1}^{\infty} ke^{-k^2/t}\le
$$
$$
2\sum_{k=1}^{\infty} 4k^2  e^{-4k^2}\le 2\left( \frac{4}{e^4}+ \sum_{n=16}^{\infty} n  e^{-n}\right) \le 0,15.
$$
Finally, using the estimates $ -0,15 <\frac{A}{1-x}<0$ and $1<B<1,02$, we get the desired result:
for $0<t\le \frac12$ and $\frac12\le x<1$ 
$$0,85< \frac{1}{1-x}A+B<1,02,
$$
which, for $\frac12\le x<1$, implies the following
$$
0,85\le
\frac{ \gamma^1(t,x) }{ \frac{1}{\sqrt{2\pi} t^{3/2}}(1-x)\frac{\sinh{\mu}}{\sinh (\mu x)}e^{-\mu^2 t/2-(1-x)^2/(2t)}}
\le 1,02.
$$
If we want to write factor $\frac{x(1-x)}{t+x}$ instead of $(1-x) $ in the denominator, we need to
multiply the right-hand side constant by 3/2, because for $0<t\le \frac14$ and $\frac12\le x<1$ we have 
$\frac23\le \frac{x}{t+x} < 1$. This gives the estimate for $t\ge \frac14$ with constants 0,85 and 1,53.
Finally, taking into account estimates for $0<x\le \frac12$ and for $\frac12\le x<1$ we get the desired estimate,
which ends the proof of the theorem.
\end{proof}
Now the estimate for $t\ge \frac14$.
\begin{theorem}\label{bigt}
For $t\ge \frac14$ and $0<x <1$
$$
0,8  \le \frac{ \gamma^1(t,x)}{\pi \sin(\pi x) \frac{ \sinh(\mu )}{\sinh(\mu x)}e^{-(\mu^2+\pi^2)t/2}} \le 
1,2.
$$
\end{theorem}
\begin{proof}
For  $t\ge \frac14$  we use formula (\ref{gamma2}) and the following inequality: for $0<x<1$ there holds $|\sin(k\pi x)|\le k \sin(\pi x)$.
The first term of the series $\sum_{k=1}^{\infty} (-1)^{k+1}k \sin(k\pi x)e^{-k^2\pi^2t/2}$
is much larger than the sum of the absolute values of all the rest:
$$
|\sum_{k=2}^{\infty} (-1)^{k+1}k \sin(k\pi x)e^{-k^2\pi^2t/2}|\le \sin(\pi x) e^{-\pi^2 t/2}\sum_{k=2}^{\infty} k^2 e^{-(k^2-1)\pi^2 t/2}\le
$$
$$
\sin(\pi x) e^{-\pi^2 t/2}\sum_{n=4}^{\infty} n e^{-(n-1)\pi^2 /8}\le
\frac{e^{-\pi^2/4}(4e^{\pi^2/8}-3)}{(e^{\pi^2/8}-1)^2}\sin(\pi x) e^{-\pi^2 t/2} \le
0,2 \sin(\pi x) e^{-\pi^2 t/2}.
$$
This means that 
$$
0,8 \sin(\pi x) e^{-\pi^2 t/2}  \le    
\sum_{k=1}^{\infty} (-1)^{k+1}k \sin(k\pi x)e^{-k^2\pi^2t/2}
\le 1,2 \sin(\pi x) e^{-\pi^2 t/2},
$$
hence for $t\ge 1/4$ and $0<x<1$
$$
 0,8 \pi\sin(\pi x ) \frac{\sinh{\mu}}{\sinh(\mu x)}e^{-\mu^2t/2-\pi^2t/2 } \le
    \gamma^{1}(t,x)  \le 1,2 \pi\sin(\pi x ) \frac{\sinh{\mu}}{\sinh(\mu x)}e^{-\mu^2t/2-\pi^2t/2 } ,
$$
which ends the proof. 
\end{proof}
Now, because for all $x\in (0,1)$ there holds $\pi^2< \frac{\pi \sin(\pi x)}{x(1-x)}<4\pi$, the above inequality implies the following:
\begin{equation}\label{corr}
 0,8 \pi^2  \le \frac{\gamma^{1}(t,x)}{x(1-x) \frac{\sinh{\mu}}{\sinh(\mu x)}e^{-\mu^2t/2-\pi^2t/2 }}  \le 4,8 \pi.
\end{equation}
Observe, that the above inequality differs from that given in Theorem \ref{smallt} for the case $0<x\le \frac12$: 
it does not contain factor $1/(\sqrt{2\pi} t^{3/2})$ in the denominator 
and instead of a factor $e^{-(1-x)^2/(2t)}$ it has a factor $e^{-\pi^2t/2}$.
If we want to have one estimate for all $t>0$ and $0<x<1$, we must add such factors. Multiplying the estimating function from Theorem \ref{smallt} by 
$e^{-\pi^2 t/2}\in (e^{-\pi^2/8},\, 1)$, we must multiply the constant on the left-hand side by $e^{-\pi^2/8} \approx 0,29$. This operation changes constant $\frac14$ from Theorem \ref{smallt} to $\frac{e^{-\pi^2/8}}{4}\approx 0,0728>0,07$.

On the other hand, the estimating function from  the denominator in (\ref{corr}) must be multiplied by  
$$\frac{(1+t)^{5/2}}{\sqrt{2\pi}(t+x)t^{3/2}} e^{-(1-x)^2/(2t)}.
$$
But for $t\ge \frac14$ and $0<x<1$ the above function is greater than $1/\sqrt{2\pi} \approx 0,3989...> 0,39$ and less than 
$\frac{ 25\sqrt{5}e^{-1/8}}{4\sqrt{2\pi}}\approx 4,92... <5$ hence constant $4,8\pi$ from (\ref{corr}) must be changed to $5\cdot 4,8\pi =24\pi$
and constant $0,8\pi^2$ from (\ref{corr}) must be changed to $0,8\pi^2/\sqrt{2\pi} > 3,1499>\pi$.

In this way theorems \ref{smallt} and \ref{bigt} together imply the following:
\begin{corollary}
 For all $t>0$ and $0<x<1$ the following inequality holds
$$
0,07 \le \frac{\gamma^1(t,x)}{\frac{x(1-x)}{t+x}\frac{\sinh(\mu)}{\sinh(\mu x)}
\frac{(1+t)^{5/2}}{\sqrt{2\pi} t^{3/2}}\exp\(-\frac{\mu^2+\pi^2}{2}t-\frac{(1-x)^2}{2t}\)} \le 24\pi <75,4.
$$
\end{corollary}

\bigskip

It is easy to notice that by (\ref{p}) the following scaling property holds: 
$$
\gamma^{r_0}(t,x)=\frac{\sinh(\mu x/r_0)}{\sinh(\mu)\sinh(\mu x)r_0^2}\exp\(-\frac{\mu}{2}t\(1-\frac1{r_0^2}\)\)\gamma^1\(\frac t{r_0^2},\frac x{r_0}\).
$$
This, together with the above Corollary proves the following.
\begin{theorem}\label{112}
\ {\it For $\mu,\, r_0,\, t>0 $ and $0<x<r_0$ we have}
\begin{equation}\label{gamma:estimates}
0,07 \le \frac{\gamma^{r_0}(t,x)}{\frac{x(r_0-x)}{t+r_0x}\frac{\sinh(\mu r_0)}{\sinh(\mu x)}\frac{(r_0^{2}+t)^{5/2}}{\sqrt{2\pi}r_0^{4}t^{3/2}}\exp\(-\frac{(r_0\mu)^2+\pi^2}{2r_0^2}t-\frac{(r_0-x)^2}{2t}\)}\le 75,4.
\end{equation}
\end{theorem}

\bigskip

\begin{rem*}
Many formulas in the book \cite{bs} are given in the language of a special function $ss_y(v,t)=\frac{1}{\sqrt{2\pi}y^{3/2}}\sum_{k=-\infty}^{\infty} (t-v+2kt)e^{-(t-v+2kt)^2/(2y)}$, $v\le t$ (see \cite{bs}, page 641). Observe that $ss_t(x,r_0)$ is precisely the series in our $\gamma^{r_0}(t,x)$ and was estimated above.
Using our method it is possible to give an estimate of $ss_t(x,r_0)$ for all possible values of variables (like we did it in Theorem \ref{112} for $\gamma^{r_0}(t,x)$) but estimates for different sets of $t$ and $x$, given in Theorems \ref{smallt} and  \ref{bigt}, are much more exact.  
For instance the proof of Theorem \ref{smallt} gives the following: for $r_0=1$, $0<t\le \frac14$ and $0<x\le \frac12$ there holds
$$
0,25  \le \frac{ ss_t(x,1)}{\frac{x}{\sqrt{2\pi t^3}(t+x)}e^{-(1-x)^2/(2t)}} \le 2,01. 
$$
\end{rem*} 

Now we estimate the density of the transition probability of the killed process.

\begin{theorem}
For fixed $r_0>0$,   
all  $x,y\in (0,\, r_0)$ and $t>0$  there holds
\begin{align*}
p^{r_0}(t;x,y)\approx&\,\frac{(r_0^2+t)^{5/2}}{r_0^5\sinh(\mu x)\sinh(\mu y)\sqrt t}\(1\wedge \frac{xy}t\)\(1\wedge \frac{(r_0-x)(r_0-y)}t\)\\
&\,\times\exp\(-\frac{(r_0\mu)^2+\pi^2}{2r_0^2}t-\frac{(x-y)^2}{2t}\).
\end{align*}
The constants in the above  estimate can be taken $c_1=0,0029  $ and  $c_2=2413$.
\end{theorem}
\begin{proof}
Observe that
\begin{equation}
p^{r_0}(t;x,y)=\frac{\sinh(\mu x/r_0)\sinh(\mu y/r_0)}{\sinh(\mu x)\sinh(\mu y)r_0}\exp\(-\frac{\mu^2t}2\left(1-\frac1{r_0^2}\right)\)
p^1\(\frac t{r_0^2};\frac{x}{r_0},\frac{y}{r_0}\),\\\label{densityhomog}
\end{equation}
thus it is sufficient to consider only the case when $r_0=1$.  Define the following function
\begin{eqnarray*}
\lambda(w)=\frac{1}{\sqrt{2\pi}t^{3/2}}\sum_{k=-\infty}^{\infty}(w+2k)\exp\(-\frac{(w+2k)^2}{2t}\),\ \ \ w\in\R.
\end{eqnarray*}
Note that the function $ss_t(0,w)$, mentioned in remark to Theorem \ref{112}, is given by the same series, only the range of its argument $w$ is different.  By the formula (\ref{density}) we have
$$
p^1(t;x,y)=\frac{e^{-\mu^2t/2}}{\sinh(\mu x)\sinh(\mu y)}\int_{x-y}^{x+y}\lambda(w)dw.
$$
From the definition of $\lambda(w)$ we can deduce the following properties of the function:
\begin{equation*}
\lambda(-w)=-\lambda(w),\ \ \  \lambda(1+w)=-\lambda(1-w),
\end{equation*}
which imply
$$p^{1}(t;x,y)=\frac{e^{-\mu^2t/2}}{\sinh(\mu x)\sinh(\mu y)}\int_{\left|x-y\right|}^{1-\left|1-x-y\right|}\lambda(w)dw.
$$
Observe that $\left|x-y\right|<1-\left|1-x-y\right|$ for $x,y\in(0,1)$. 
Since $\gamma^{1}(t,x)=\frac{\sinh(\mu)e^{-\mu^2t/2}}{\sinh(\mu x)}\lambda(1-x)$ (cf. \ref{p})), 
we get by virtue of Corollary 1:
\begin{equation}\label{eq:lambdaest}\lambda(1-x)\approx \frac{x(1-x)}{t+x}\frac{(1+t)^{5/2}}{t^{3/2}}\exp\(-\frac{\pi^2}{2}t-\frac{(1-x)^2}{2t}\).
\end{equation}
Consequently
$$p^{1}(t;x,y)\approx\frac{e^{-(\mu^2+\pi^2)t/2}(1+t)^{5/2}}{\sinh(\mu x)\sinh(\mu y)t^{3/2}}
\int_{\left|x-y\right|}^{1-\left|1-x-y\right|}\frac{w(1-w)}{t+(1-w)}e^{-w^2/2t}dw.
$$
Now, substituting $w=\sqrt{1-vt}$, we get
$$1-w=1-\sqrt{1-vt}=\frac{vt}{1+\sqrt{1-vt}} \ \ \mbox{and} \ \ dw=\frac{-t}{\sqrt{1-vt}},$$ 
hence we have
\begin{equation}\label{eq:constinsubst}
\frac12tv\le 1-w\le tv.
\end{equation}
Combining this with Lemma \ref{lemma:int1} from the Appendix,  we obtain 
\begin{eqnarray*}p^{1}(t;x,y)&\approx&\frac{e^{-(\mu^2+\pi^2)t/2}(1+t)^{5/2}}{\sinh(\mu x)\sinh(\mu y)t^{1/2}}e^{-1/(2t)}
\int^{\frac{1-(x-y)^2}t}_{\frac{1-(1-\left|1-x-y\right|)^2}t}\frac{v}{1+v}e^{v/2}dv\\
&\approx &\frac{e^{-t(\mu^2+\pi^2)/2}(1+t)^{5/2}}{\sinh(\mu x)\sinh(\mu y)t^{1/2}}e^{-(x-y)^2/2t}\\
&&\times\(1\wedge\frac{1-(x-y)^2}{t}\)\(1\wedge\frac{(1-\left|1-x-y\right|)^2-(x-y)^2}{t}\).
\end{eqnarray*}
We rewrite the expression in the last parentheses as follows
\begin{eqnarray*}
(1-\left|1-x-y\right|)^2-(x-y)^2&=&\left\{\begin{array}{ll}4xy&y+x<1,\\4(1-x)(1-y)&y+x\ge1,\end{array}\right.\\[12pt]
&=&4[xy\wedge((1-x)(1-y))].
\end{eqnarray*}
Moreover, it is easy to check, that $1\le\frac{1-(x-y)^2}{(1-x)(1-y)}\le4$ for $0\le x\le 1- y\le 1$. Since $f(x,y)=1-(x-y)^2=f(1-x,1-y)$, we get
for all $x,\, y\in (0,\, 1)$
 $$1\le\frac{1-(x-y)^2}{(xy)\vee((1-x)(1-y))}\le4.$$
This implies that for all $x,\, y\in (0,\, 1)$
\begin{align*}
\(1\wedge\frac{1-(x-y)^2}{t}\)&\(1\wedge\frac{(1-\left|1-x-y\right|)^2-(x-y)^2}{t}\)\\
&\le16\(1\wedge\frac{(xy)\vee(1-x)(1-y))}{t}\)\(1\wedge\frac{(xy)\wedge((1-x)(1-y))}t\)\\
&=16\(1\wedge\frac{xy}{t}\)\(1\wedge\frac{(1-x)(1-y)}{t}\),
\end{align*}
and similarly
\begin{align*}
\(1\wedge\frac{1-(x-y)^2}{t}\)\(1\wedge\frac{(1-\left|1-x-y\right|)^2-(x-y)^2}{t}\)\ge\(1\wedge\frac{xy}{t}\)\(1\wedge\frac{(1-x)(1-y)}{t}\).
\end{align*}
Theorem \ref{112} implies that in the estimate (\ref{eq:lambdaest}) of $\lambda(x)$ we have constants $c_1=0,07$ and $c_2=75,4$. 
Using this, the above inequalities, constants from Lemma 1 and inequality (\ref{eq:constinsubst}), we get the constants in the thesis of the theorem. 
\end{proof}

Recall that $M_t=\sup_{0<s\le t} Z_s$ and denote $m(t,x,y)=\p^x(M_t\in dy)$.
An estimate of this density is the most complicated. In the proof we will use three elementary lemmas, which will be proved in the Appendix.

\begin{theorem}
For all $0<x<y$ and $t>0$ the following estimate holds:
$$m(t;x,y)\approx\frac{x(y-x)}{y^{2}t}\frac{\sinh(\mu y)}{\sinh(\mu x)}\frac{(y^2+t)^{5/2}\(1+\frac{t^{3/2}}{y^3}+\frac{\sqrt t}{y-x}+\sqrt t\mu\)}{\(y^2+t((y\mu)^2+1)\)(t+yx)}\frac{e^{-(y-x)^2/2t-\frac{(y\mu)^2+\pi^2}{2y^2}t}}{\sqrt{1+\frac{(y-x)^2y^2}{t(y^2+t((y\mu)^2+1))}}}.$$
\end{theorem}
\begin{proof}
By the strong Markov property (SMP) we have for $y'>y>x$
\begin{eqnarray*}
\p^x\(M_t\in(y,y')\)&=&\p^x\(\tau_y<t,\tau_{y'}>t\)\\
&\stackrel{SMP}{=}&\mathbb{E}^x\[\tau_y<t;\mathbb{E}^{y}\(\tau_{y'}>t-\tau_y\)\]\\
&=&\int_0^t\gamma^{y}(u,x)\int_{t-u}^\infty \gamma^{y'}(v,y)\,dv\,du.
\end{eqnarray*}
By (\ref{gamma:estimates}) we obtain
\begin{align*}\p^x\(M_t\in(y,y')\)\approx& \frac{x(y-x)y(y'-y)}{(y'y)^4}\frac{\sinh(\mu y')}{\sinh(\mu x)}\\
&\times\int_0^t\frac{(y^{2}+u)^{5/2}}{(u+yx)u^{3/2}}e^{-\frac{(y\mu)^2+\pi^2}{2y^2}u-\frac{(y-x)^2}{2u}}\int_{t-u}^\infty \frac{(y'^2+v)^{5/2}}{(v+y'y)v^{3/2}}e^{-\frac{(y'\mu)^2+\pi^2}{2y'^2}v-\frac{(y'-y)^2}{2v}}\,dv\,du\\
:=&F(t,x;y,y').
\end{align*}
Since $m(t;x,y)$ is continuous (it is obtained by differentiating formula (\ref{supremum}) with respect to $r_0$), then $m(t;x,y)=\lim_{y'\rightarrow y}\frac{\p^x\(M_t\in(y,y')\)}{y'-y}$, hence
\begin{align*}
m(t;x,y)\approx&\lim_{y'\rightarrow y}\frac{F(t,x;y,y')}{y'-y}\\
=&\frac{x(y-x)}{y^{7}}\frac{\sinh(\mu y)}{\sinh(\mu x)}\\
&\times\int_0^t\frac{(y^{2}+u)^{5/2}}{(u+yx)u^{3/2}}e^{-\frac{(y\mu)^2+\pi^2}{2y^2}u-\frac{(y-x)^2}{2u}}\int_{t-u}^\infty \frac{(y^2+v)^{3/2}}{v^{3/2}}e^{-\frac{(y\mu)^2+\pi^2}{2y^2}v}\,dv\,du\\
=&\frac{x(y-x)}{y^{7}}\frac{\sinh(\mu y)}{\sinh(\mu x)}e^{-\frac{(y\mu)^2+\pi^2}{2y^2}t}\\
&\times\int_0^t\frac{(y^{2}+u)^{5/2}}{(u+yx)u^{3/2}}e^{-\frac{(y-x)^2}{2u}}\int_{0}^\infty \frac{(y^2+t-u+v)^{3/2}}{(t-u+v)^{3/2}}e^{-\frac{(y\mu)^2+\pi^2}{2y^2}v}\,dv\,du.
\end{align*}
We apply Lemma 2 from Appendix with $a=y^2$, $b=t-u$, $c=\frac{(y\mu)^2+\pi^2}{2y^2}$ to the inner integral and get
$$m(t;x,y)\approx\frac{x(y-x)}{y^{5}}\frac{\sinh(\mu y)}{\sinh(\mu x)}e^{-\frac{(y\mu)^2+\pi^2}{2y^2}t}\int_0^t\frac{(y^{2}+u)^{5/2}(y^2+t-u)^{3/2}e^{-\frac{(y-x)^2}{2u}}}{\sqrt {t-u}(u+yx)u^{3/2}\(y^2+(t-u)((y\mu)^2+1)\)}du.
$$
We split the above-given integral into two parts $\int_0^{t/2}+\int_{t/2}^t=I_1+I_2$ and substitute $w=\frac{(y-x)^2}{2}(\frac1u-\frac2t)$, $w=\frac{(y-x)^2}{2}(\frac1u-\frac1t)$ in $I_1$, $I_2$, respectively. Before substitution we use the following estimates:
$t-u\approx t$  for $u\in\(0,\frac t2\)$ and $u\approx t$  for $u\in\(\frac t2, t\)$. Additionally, the second substitution gives $t-u=2t^2w/(2tw+(y-x)^2)$. Consequently
\begin{align*}
I_1\approx&\,\frac{(y^2+t)^{3/2}}{\sqrt {t}\(y^2+t((y\mu)^2+1)\)}\int_0^{t/2}\frac{(y^{2}+u)^{5/2}e^{-\frac{(y-x)^2}{2u}}}{(u+yx)u^{3/2}}du
\\=&\,\frac{\sqrt2y^4(y^2+t)^{3/2}e^{-(y-x)^2/t}}{\sqrt tx(y-x)\(y^2+t((y\mu)^2+1)\)}\int_0^\infty\frac{\(w+(y-x)^2(\frac1t+\frac1{2y^2})\)^{5/2}e^{-w}}{\(w+\frac{(y-x)^2}t\)^2\(w+(y-x)^2(\frac1t+\frac1{2xy})\)}dw,\\
I_2\approx&\,\frac{(y^{2}+t)^{5/2}}{(t+yx)t^{3/2}}\int_{t/2}^t\frac{(y^2+t-u)^{3/2}e^{-\frac{(y-x)^2}{2u}}}{\sqrt {t-u}\(y^2+(t-u)((y\mu)^2+1)\)}du\\
=&\,\frac{2(y^{2}+t)^{4}e^{-(y-x)^2/2t}}{(y-x)^2(t+yx)(y^2+t(y\mu)^2+1)}\int_0^{(y-x)^2/2t}\frac{(w+\frac{(y-x)^2y^2}{2t(y^2+t)})^{3/2}e^{-w}}{\sqrt w\(w+\frac{(y-x)^2y^2}{2t(y^2+t((y\mu)^2+1))}\)}dw.\\
\end{align*}
Now we apply Lemma 3 to the integral in the estimate of $I_1$ with $a=(y-x)^2(\frac{1}{t}+\frac1{2y^2})$, $b=\frac{(y-x)^2}{t}$, $c=(y-x)^2(\frac1t+\frac1{2yx})$ and Lemma 4 to the integral in the estimate of $I_2$ with $a=\frac{(y-x)^2}{2t}$, $b=\frac{(y-x)^2y^2}{2t(y^2+t)}$, $c=\frac{(y-x)^2y^2}{2t(y^2+t((y\mu)^2+1))}$. For $\frac{(y-x)^2}{2t}\ge1$ we have $y^2\ge t$ which implies
$$\frac{(y-x)^2y^2}{2t(y^2+t)}=\frac{(y-x)^2}{t}\frac{1}{2(1+\frac{t}{y^2})}\approx \frac{(y-x)^2}{t},$$
so all assumptions of Lemma 4 are satisfied. Hence
\begin{align*}
I_1+I_2\approx&\,\frac{(y^2+t)^{3/2}y^4e^{-(y-x)^2/t}}{\sqrt {t}\(y^2+t((y\mu)^2+1)\)x(y-x)}\(\frac{t(y-x)(\frac{y^2+t}{ty^2})^{5/2}}{(1+\frac{(y-x)^2}{t})(\frac{yx+t}{yxt})}+\frac1{1+(y-x)^2(\frac{yx+t}{yxt})}\)\\
&\,+\frac{(y^{2}+t)^{4}e^{-(y-x)^2/2t}}{(y-x)^2(t+yx)(y^2+t((y\mu)^2+1))}\times\\
&\,\times\(1\wedge\frac{(y-x)^2}{2t}+\frac{\sqrt{t(y^2+t((y\mu)^2+1))}(y-x)^2y^2}{((y^2+t)t)^{3/2}\sqrt{1+\frac{(y-x)^2y^2}{t(y^2+t((y\mu)^2+1))}}}\)\\
=&\,\frac{(y^2+t)^{3/2}y^4e^{-(y-x)^2/2t}}{\sqrt {t}\(y^2+t((y\mu)^2+1)\)(t+yx)}\(\frac{(y^2+t)^{5/2}e^{-(y-x)^2/2t}}{\sqrt ty^4(1+\frac{(y-x)^2}{t})}+\frac{\frac{t+yx}{x(y-x)}e^{-(y-x)^2/2t}}{1+(y-x)^2(\frac{yx+t}{yxt})}+\right.\\&\,+\left.\(1\wedge\frac{(y-x)^2}{2t}\)\frac{(y^2+t)^{5/2}\sqrt t}{(y-x)^2y^4}+\frac{(y^2+t)\sqrt{y^2+t((y\mu)^2+1)}}{ \sqrt ty^2\sqrt{1+\frac{(y-x)^2y^2}{t(y^2+t((y\mu)^2+1))}}}\).
\end{align*}
Let us denote the expression in the last brackets by $\mathcal J(y,x,t,\mu)$. To prove the theorem we need to show that 
\begin{equation}\label{J}
\mathcal J(y,x,t,\mu)\approx\frac{(y^2+t)}{y\sqrt t}\frac{1+\frac{t^{3/2}}{y^3}+\frac{\sqrt t}{y-x}+\sqrt t\mu  }{\sqrt{1+\frac{(y-x)^2y^2}{t(y^2+t((y\mu)^2+1))}}}.
\end{equation}
Assume  $\frac{(y-x)^2}{t}>1$. Then we have $y^2>t$, which implies
\begin{align*}
\frac{(y^2+t)^{5/2}e^{-(y-x)^2/2t}}{\sqrt ty^4(1+\frac{(y-x)^2}{t})}&<2\frac{(y^2+t)^{3/2}}{\sqrt ty^2\sqrt{1+\frac{(y-x)^2}{t}}}<2\frac{(y^2+t)\sqrt{y^2+t((y\mu)^2+1})}{\sqrt ty^2\sqrt{1+\frac{(y-x)^2y^2}{ t(y^2+t((y\mu)^2+1))}}},\\
\frac{\frac{t+yx}{x(y-x)}e^{-(y-x)^2/2t}}{1+(y-x)^2(\frac{yx+t}{yxt})}&<\frac{\frac{\sqrt t}{(y-x)}y}{\sqrt t\frac{(y-x)^2}t}\(\frac{y^2+t}{y^2}\)<2\frac{(y^2+t)\sqrt{y^2+t((y\mu)^2+1})}{\sqrt ty^2\sqrt{1+\frac{(y-x)^2y^2}{ t(y^2+t((y\mu)^2+1))}}},\\
\(1\wedge\frac{(y-x)^2}{2t}\)\frac{(y^2+t)^{5/2}\sqrt t}{(y-x)^2y^4}&<2\frac{(y^2+t)^{3/2}}{\sqrt ty^2\frac{y-x}{\sqrt t}}<2\frac{(y^2+t)\sqrt{y^2+t((y\mu)^2+1})}{\sqrt ty^2\sqrt{1+\frac{(y-x)^2y^2}{ t(y^2+t((y\mu)^2+1))}}},
\end{align*}
so we can estimate $\mathcal J(y,x,t,\mu)$ from the above by its last component.
But for $x,y,z >0$ there holds $\sqrt{x^2+y^2+z^2}\approx x+y+z$, because $l_2$ and $l_1$ norms are equivalent in $\R^3$.
Hence $\frac{\sqrt{y^2+t((y\mu)^2+1)}}{y} \approx 1+\sqrt t\mu+\frac{\sqrt t}{y} $ and 
because $(y-x)^2>t$, then (\ref{J}) follows. 

In the case $\frac{(y-x)^2}{t}\le1$ we have $t+xy\approx t+y^2$ and also (recall that $0<x<y$)
 $$1+(y-x)^2\frac{yx+t}{yxt}=\frac yx+\frac xy+\frac{(y-x)^2}{t}-1\approx\frac yx.$$
Hence
\begin{align*}
\mathcal J(y,x,t,\mu)\approx&\, \frac{(y^2+t)^{5/2}}{\sqrt ty^4}+\frac{t+y^2}{y(y-x)}+\frac{(y^2+t)^{5/2}}{\sqrt ty^4}+\frac{(y^2+t)\sqrt{y^2+t((y\mu)^2+1)}}{ \sqrt ty^2}\\
=&\,\frac{y^2+t}{y\sqrt t}\(\frac{2(y^2+t)^{3/2}}{y^3}+\frac{\sqrt t}{y-x}+\frac{\sqrt{y^2+t((y\mu)^2+1)}}{ y}\)\\
\approx&\,\frac{y^2+t}{y\sqrt t}\(1+\frac{t^{3/2}}{y^3}+\frac{\sqrt t}{y-x}+\sqrt t\mu+\frac{\sqrt t}y\),
\end{align*}
which again is equivalent to (\ref{J}).
\end{proof}

\section{Appendix}
Here we gathered four lemmas which were used in the estimates carried out in the previous section.
\begin{lemma}\label{lemma:int1}
For $0<a<b$ we have
\begin{align*}\frac1{12}\le\frac{\int_a^b\frac{w}{1+w}e^{w/2}dw}{e^{b/2}(1\wedge b) (1\wedge(b-a))}\le 2.
\end{align*}
\end{lemma}
\begin{proof}
Let us denote $\mathcal I(a,b)=\int _a^b\frac{w}{1+w}e^{w/2}dw$ and let $a,b\le1$. 
For $w\in(0,\, 1)$ the function $f(w)=\frac{e^{w/2}}{1+w}$ is decreasing hence for such $w$ there holds
$\frac{\sqrt{e}}{2}\le f(w)\le 1$. Using these inequalities we get
$$
\frac{\sqrt{e}}{2} \int _a^b w \,dw \le \int _a^b\frac{w}{1+w}e^{w/2}\,dw \le  \int _a^b w \,dw=b^2-a^2=(b-a)(b+a).
$$
Since $0<a<b\le 1$, we have  $b<b+a<2b$ and this implies
$$
\frac{1}{2}e^{b/2}b(b-a)\le\frac{\sqrt{e}}{2}b(b-a)
\le \mathcal I(a,b) \le \int _a^b w\,dw  \le 2b(b-a)\le 2e^{b/2}b(b-a).$$
Now let  $b>1$. We use inequality $1-e^{-x}\le1\wedge x,$ $x\ge0$, and get
$$\mathcal I(a,b)\le\int_a^be^{w/2}dw=e^{b/2}\(1-e^{-(b-a)/2}\)\le e^{b/2}\(1\wedge(b-a)\).$$
For $x\ge0$ it holds that $1-e^{-x}\ge\frac{x}{1+x}\ge\frac12(1\wedge x)$. Using this, the fact that 
the function $x\rightarrow \frac x{1+x}$ is increasing  and $(a+b)/2 \ge 1/2$, we obtain
\begin{align*}
\mathcal I(a,b)\ge&\,\int _{\frac{a+b}2}^b\frac{w}{1+w}e^{w/2}dw\ge\frac{\frac{a+b}2}{1+\frac{a+b}2}\int _{\frac{a+b}2}^be^{w/2}dw\ge \frac{\frac{1}2}{1+\frac{1}2}2e^{b/2}\(1-e^{-(b-a)/4}\)\\
\ge&\,\frac{1}{12}e^{b/2}\(1\wedge(b-a)\),
\end{align*}
which ends the proof.
\end{proof}
\begin{lemma}\label{lemma:int2}
For $a,b,c>0$ and $ac>1$ we have 
\begin{equation}\label{int2}\int_{0}^\infty \frac{(a+b+v)^{3/2}}{(b+v)^{3/2}}e^{-cv}\,dv\approx\frac{(a+b)^{3/2}}{\sqrt b(1+bc)}.\end{equation}
\end{lemma}
\begin{proof}
Let us denote the integral in the thesis by $I(a,b,c)$. Substituting $v=\frac sc$ we obtain
\begin{equation}\label{Iabc}I(a,b,c)=\frac1c\int_{0}^\infty \frac{((a+b)c+s)^{3/2}}{(bc+s)^{3/2}}e^{-s}\,ds.\end{equation}
For $bc\ge1$ we get
$$I(a,b,c)=\frac1c \left(\frac{a+b}{b}\right)^{3/2}\int_{0}^\infty \frac{(1+s/((a+b)c))^{3/2}}{(1+s/(bc))^{3/2}}e^{-s}\,ds \approx \frac1c\(\frac{a+b}{b}\)^{3/2},$$
which is equivalent to  \ref{int2}. For $bc<1$ we split the integral in (\ref{Iabc}) as follows: $\frac{1}{c}\int_0^1+\frac{1}{c}\int_1^\infty=I_1+I_2$. Then 
$$I_2\approx \sqrt c{(a+b)^{3/2}}.$$
In $I_1$  we substitute $s=bc\,t$ and get
$$
I_1\approx \,b\int_{0}^{1/bc} \frac{\(\frac{a+b}b\)^{3/2}+t^{3/2}}{(1+t)^{3/2}}e^{-bct}\,dt 
\approx \,b\(\frac{a+b}b\)^{3/2}+\frac1c,
$$
hence
\begin{align*}
I_1+I_2\approx\frac1c\(bc\(\frac{a+b}b\)^{3/2}+1+(bc)^{3/2}\(\frac{a+b}b\)^{3/2}\).
\end{align*}
From the assumptions $bc<1$ and $ac>1$ it follows  that
$$bc\(\frac{a+b}b\)^{3/2}>(bc)^{3/2}\(\frac{a+b}b\)^{3/2}>1,$$
which ends the proof of the lemma.
\end{proof}

\begin{lemma}\label{lemma:int3}
Let $a,c>b>0$ and $a<2$ if $b<1$. Then we have 
\begin{equation}\label{int3}
\int_0^\infty\frac{(w+a)^{5/2}}{(w+b)^2(w+c)}e^{-w}dw\approx\frac{a^{5/2}}{b(b+1)c}+\frac{1}{1+c}.
\end{equation}
\end{lemma}
\begin{proof}
Denote the integral in the thesis by $F(a,b,c)$. It is clear that for $b\ge1$ we have $F(a,b,c)\approx\frac{a^{5/2}}{b^2c}$, which is equivalent to (\ref{int3}). For $b<1$ we split the integral into two parts
$$F(a,b,c)\approx a^{5/2}\int_0^\infty \frac{1}{(w+b)^2(w+c)}e^{-w}dw+\int_0^\infty\frac{w^{5/2}}{(w+b)^2(w+c)}e^{-w}dw.$$
The last integral can be approximated by $\frac1{1+c}$, which is shown below: 
\begin{align*}
\frac1{1+c}\approx\int_0^\infty\frac{w^{3/2}}{(w+1)^2}\frac{w}{w+c}e^{-w}dw&\le \int_0^\infty\frac{w^{5/2}}{(w+b)^2(w+c)}e^{-w}dw\\
&\le \int_0^\infty\frac{1}{\sqrt w}\frac{w}{w+c}e^{-w}dw\approx\frac1{1+c}.
\end{align*}
Moreover,
\begin{align*}
 \int_0^\infty \frac{1}{(w+b)^2(w+c)}e^{-w}dw=&\int_0^1+\int_1^\infty=\\
=&\frac{1}{bc}\int_0^{1/b}\frac{e^{-bu}du}{(1+u)^2(1+\frac{b}{c}u)}+\int_0^\infty\frac{e^{-(w+1)}dw}{(w+b+1)^2(w+c+1)}\\
\approx&\frac1{bc}+\frac1{1+c}.
\end{align*}
Hence for $b<1$ (in this case $a<2$ by assumption) we obtain
$$F(a,b,c)\approx a^{5/2}\(\frac1{bc}+\frac1{1+c}\)+\frac1{1+c}\approx \frac{a^{5/2}}{bc}+\frac{1}{1+c},$$
which is again equivalent to (\ref{int3}).
\end{proof}

\begin{lemma}\label{lemma:int4}
Let $a>b>c>0$. Assume also that $(\star)$ $a\approx b$ if $a\ge1$. Then we have
\begin{equation}
\int_0^a\frac{(w+b)^{3/2}e^{-w}dw}{\sqrt w (w+c)}\approx (1\wedge a)+\frac{b^{3/2}}{\sqrt{c(1+c)}}.
\end{equation}
\end{lemma}

\begin{proof}
Let us denote the integral by $I(a,b,c)$. Assume now that $c\ge1$. Then we have $a,b\ge1$ and get 
$$I(a,b,c)=\frac{b^{3/2}}{c}\int_0^a\frac{\(\frac{w}b+1\)^{3/2}e^{-w}dw}{\sqrt w \(\frac wc+1\)}\approx \frac{b^{3/2}}{c}.$$
It is equivalent to the thesis because $\frac{\sqrt{2}b^{3/2}}{\sqrt{c(1+c)}}>\frac{b^{3/2}}{c}>b^{1/2}>1$. In the case  $c<1$ we substitute  $w=cu$:
\begin{align}\label{wingint}
I(a,b,c)=\,c\int_0^{a/c}\frac{\(u+\frac bc\)^{3/2}e^{-cu}}{\sqrt u (u+1)}du.
\end{align}
For $a\ge1$ we use the assumption $(\star)$. It follows that $u+\frac bc\approx \frac bc$ for $u\in[0,\frac ac]$. Hence
\begin{align*}
I(a,b,c) \approx\,c\(\frac bc\)^{3/2}\int_0^{a/c}\frac{e^{-uc}}{\sqrt u (u+1)}du\approx \frac {b^{3/2}}{\sqrt{c}},
\end{align*}
which, as before, is equivalent to the thesis. For $a<1$ the formula (\ref{wingint}) gives us
\begin{align*}
I(a,b,c)\approx\,c\int_0^{a/c}\frac{udu}{u+1}+\frac{b^{3/2}}c\int_0^{a/c}\frac{e^{-uc}du}{\sqrt u (u+1)}du.
\end{align*}
The last integral is bounded by $\int_0^{1}\frac{e^{-u}du}{\sqrt u (u+1)}$ from below and by $\int_0^{\infty}\frac{du}{\sqrt u (u+1)}$ from above. Moreover
\begin{align*}
c\int_0^{a/c}\frac{udu}{u+1}= c\(\frac ac-\ln\(1+\frac ac\)\)\approx a,
\end{align*}
which ends the proof.
\end{proof}

\end{document}